\documentclass{amsart}
\usepackage{graphicx}
\usepackage{amsmath}
\usepackage{amscd}
\usepackage{amsfonts}
\usepackage{amssymb}
\usepackage[mathscr]{eucal}

\numberwithin{equation}{section}

\newtheorem{theorem}{Theorem}[section]
\newtheorem{lemma}[theorem]{Lemma}
\newtheorem{proposition}[theorem]{Proposition}

\newtheorem{definition}[theorem]{Definition}

\newtheorem{remark}[theorem]{Remark}

\newcommand{\be}{\begin{equation}}
\newcommand{\ee}{\end{equation}}
\newcommand{\bes}{\begin{equation*}}
\newcommand{\ees}{\end{equation*}}

\newcommand{\cA}{\mathcal{A}}
\newcommand{\cB}{\mathcal{B}}

\newcommand{\sB}{\mathscr{B}}

\newcommand{\mb}[1]{\mathbb{#1}}

\newcommand{\spn}{\operatorname{span}}

\begin{document}

\title[Three commuting unital CP maps with no dilation]{Three commuting, unital, completely positive maps that have no minimal dilation}

\author{Orr Moshe Shalit}
\address{Department of Pure Mathematics, University of Waterloo.}
\email{oshalit@uwaterloo.ca}
\author{Michael Skeide}
\address{Dipartimento S.E.G.e S.Universit\`a degli Studi del Molise.}
\email{skeide@unimol.it}
\keywords{Product system, subproduct system, semigroups of completely positive maps, dilation.}
\subjclass[2000]{46L55.}
\date{\today}

\begin{abstract}
{In this note we prove that there exist at least two examples of three commuting, unital, completely positive maps that have no dilation on a type I factor, and no minimal dilation on any von Neumann algebra.}
\end{abstract}

\maketitle
\section{Introduction}

\subsection{Definitions and statement of the main result} 
Let $\cA \subseteq \sB(H)$ be a von Neumann algebra. Let $\mb{S}$ be a a commutative semigroup with unit $0$. A \emph{CP-semigroup} (over $\mb{S}$) on $\cA$ is a family $T = \{T_s \}_{s \in \mb{S}}$ such that 
\begin{enumerate}
\item For all $s \in \mb{S}$, $T_s$ is a normal, contractive completely positive map (henceforth: CP map).
\item $T_0 = {\bf id}_{\cA}$.
\item For all $s,t \in \mb{S}$, 
\bes
T_s \circ T_t = T_{s+t}.
\ees
\end{enumerate}

A CP-semigroup $T$ is called a {\em Markov semigroup} (or a {\em CP$_0$-semigroup}) if $T_s$ is unital (i.e., $T_s(I_H) = I_H$) for every $s$. Every $k$-tuple of commuting CP maps gives rise naturally to a CP-semigroup over $\mb{N}^k$, and vice-versa. 

If every element of $T$ is a $*$-endomorphism, then $T$ is said to be an {\em E-semigroup}. An E-semigroup in which every element is unital is called an {\em E$_0$-semigroup}.

\begin{definition}\label{def:dilArv}
Let $T = \{T_s\}_{s\in\mb{S}}$ be a CP-semigroup on $\cA$. A {\em dilation} for $T$ is a quadruple $(p,K,\cB,\theta)$, where $K$ is a Hilbert space, $\cB \subseteq \sB(K)$ is a von Neumann algebra, $p \in \cB$ is the orthogonal projection $K\rightarrow H$, $\cA$ is contained in $\cB$ as a corner $\cA = p \cB p$, and $\theta = \{\theta_s\}_{s\in\mb{S}}$ is an E-semigroup on $\cB$ such that for all $s\in \mb{S}$, 
\be\label{eq:dilweaker}
T_s(a) = p \theta_s(a) p \quad , \quad \textrm{for all } a \in \cA.
\ee
\end{definition}

\begin{definition}
Let $T = \{T_s\}_{s\in\mb{S}}$ be a CP-semigroup on $\cA$. A {\em strong dilation} for $T$ is a dilation $(p,K,\cB,\theta)$ for $T$ such that for all $s\in \mb{S}$, 
\be\label{eq:dilArv}
T_s(pbp) = p \theta_s(b) p \quad , \quad \textrm{for all } b \in \cB.
\ee
\end{definition}

\begin{remark}\emph{
When $T$ is unit preserving, the condition (\ref{eq:dilArv}) is equivalent to the seemingly weaker condition (\ref{eq:dilweaker}). Indeed, if $T$ is unit preserving and $(p,K,\cB,\theta)$ is a dilation then $p = T_s(p) = p\theta_s(p) p$, thus $\theta_s(p) \geq p$. It follows that 
$$p\theta_s(b)p = p \theta_s(p) \theta_s(b) \theta_s(p) p = p\theta_s(pbp)p = T_s(pbp) ,$$
whence $(p,K,\cB,\theta)$ is a strong dilation. In general  (\ref{eq:dilArv}) is stronger than (\ref{eq:dilweaker}), hence the terminology. Some authors refer to a strong dilation as an \emph{E-dilation}.} 
\end{remark}
This notion of dilation has received much attention in the last two decades. It has been proved that a strong dilation exists for single CP maps and for one-parameter CP-semigroups \cite{Bhat96,SeLegue,BS00,MS02,Arv03}, for a pair of commuting CP maps \cite{Bhat98,Solel06} and for some two-parameter CP-semigroups \cite{Shalit08,Shalit09}. The result that a dilation exists for one-parameter semigroups is originally due to Bhat \cite{Bha99}.

In \cite[Section 5.3]{ShalitSolel} examples were given of three-tuples of commuting CP maps that have no (minimal) strong dilation. That raised the problem of finding sufficient conditions for a commuting $k$-tuple of CP maps to have a strong dilation. There were good reasons (which will be discussed below) to conjecture that every commuting $k$-tuple of {\em unital} CP maps has a (strong) dilation\footnote{The first named author has previously announced -- wrongly -- that this is true.}. In this paper we show that this conjecture fails to a very large extent.

To state our result precisely we need one more definition, due to Arveson (see Lemma \ref{lem:min_equiv} below for an equivalent definition).

\begin{definition}
In the notation of Definition \ref{def:dilArv}, $(p,K,\cB,\theta)$ is said to be {\em minimal} if the $W^*$-algebra generated by $\cup_{s\in\mb{S}} \theta_s(\cA)$ is equal to $\cB$, and if the central carrier of $p$ in $\cB$ (the minimal projection $q \in \cB \cap \cB'$ such that $qp = p$) is equal to $1_{\cB}$.
\end{definition}

We warn the reader that there are other natural definitions of {\em minimal dilation} that could be used, see for example \cite[Sections 8.3 and 8.9]{Arv03} or \cite{Bhat03}. In the one-parameter unit-preserving case, the above definition is equivalent to some of the other natural definitions. 

The main result of this note is the following:
\begin{theorem}\label{thm:main}
There exists a Hilbert space $H$ and three commuting unital CP maps on $\sB(H)$ that have no minimal dilation, and no dilation whatsoever (minimal or not) acting on a type I factor. One can take $\dim H = 6$.
\end{theorem}

\subsection{Why is our main result interesting?}
The reader who is familiar with the theory of isometric dilations of contractions on a Hilbert space 
\cite{SzNF70}, and with Parrot's example of three commuting contractions that have no isometric dilation \cite{Parrot}, might think that Theorem \ref{thm:main} is obvious, or at least that it should have been expected. Let us explain the element of surprise in our result. The reader should be aware that the following paragraphs are of a heuristic nature. 

The prototypical example of a normal, completely positive map $T$ on $\sB(H)$ is of the form
\be\label{eq:analog}
T(a) = tat^* \,\, , \,\, a \in \sB(H)
\ee
where $t$ is a contraction. Some properties of $t$ are reflected faithfully by $T$: $T$ is a $*$-endomorphism if and only if $t$ is an isometry; $T$ is unital if and only if $t$ is a coisometry; $\|T\| = \|t\|^2$, so in particular $T$ is contractive if and only if $t$ is. Also, if $\{t_s\}_{s\in\mb{S}}$ is a semigroup of contractions, then $T_s (\bullet) = t_s \bullet t_s^*$ is a CP-semigroup. In particular, if $t_1, \ldots, t_k$ is a commuting $k$-tuple of contractions then the corresponding $T_1, \ldots, T_k$ are a commuting $k$-tuple of CP maps.

Let $t = \{t_s\}_{s\in\mb{S}}$ be a semigroup of contractions on $H$. Recall that a minimal isometric dilation for $t$ is a semigroup of isometries $v = \{v_s\}_{s \in \mb{S}}$ on a Hilbert space $K \supseteq H$ such that
\bes
t_s = p_H v_s\big|_H \,\, , \,\, s \in \mb{S},
\ees
and
\bes
K = \overline{\spn}\{v_s h : h \in H, s \in \mb{S}\} .
\ees
It follows that
\bes
v_s ^* H \subseteq H  \,\, , \,\, s \in \mb{S}.
\ees
One can check that if $t = \{t_s\}_{s\in\mb{S}}$ is a semigroup of contractions on $H$ and $T_s(\bullet) = t_s \bullet t_s^*$ is the corresponding CP-semigroup, then a minimal dilation $v = \{v_s\}_{s \in \mb{S}}$ (acting on $K$) gives rise to a strong dilation $\theta$ for $T$ by 
\bes
\theta_s(b) = v_s b v_s^* \,\, , \,\, b \in \sB(K).
\ees
There is also some kind of a converse -- see Section 5 in \cite{ShalitSolel}.

The development of the dilation theory of CP-semigroups followed the line of development of the well known (see \cite{SzNF70}) dilation theory of contractions. Every one-parameter semigroup of contractions (continuous or discrete) has an isometric dilation -- and every one-parameter CP-semigroup has a strong dilation \cite{Bhat96,SeLegue,BS00,MS02,Arv03}. Every pair of commuting contractions has an isometric dilation -- and every pair of commuting CP maps has a strong dilation \cite{Bhat98,Solel06}. There are some partial results regarding strong dilations for two-parameter CP-semigroups \cite{Shalit08,Shalit09}, corresponding to the existence of isometric dilations for two-parameter semigroups of contractions \cite{Slocinski}. There exist three-tuples of commuting contractions that have no isometric dilations -- and it was shown that there exist three-tuples of commuting CP maps with no strong dilation \cite{ShalitSolel}. One gets the impression that the theories are synchronized. It might be hoped that (\ref{eq:analog}) could serve as a guide for translating dilation theorems for contractive semigroups to dilation theorems for CP-semigroups. 

However, there do exist simple sufficient conditions that ensure that a $k$-tuple of commuting contractions has an isometric dilation. For example, 

\vskip 0.2cm

\noindent{\bf Theorem.} (See \cite[Theorem I.9.2]{SzNF70}) {\em
Every $k$-tuple of commuting coisometries has an isometric dilation. }

\vskip 0.2cm

In (\ref{eq:analog}) coisometric operators correspond to unital maps. This is the reason, taking into account all that the two dilation theories have in common, that it was natural to conjecture that every $k$-tuple of commuting unital CP maps has a (strong) dilation. However, our main result is that there exist a three-tuple of commuting unital CP maps with no dilation, and this shows that the dilation theory of CP-semigroups has some subtleties which do not occur in the classical theory.

One might think that the additional subtleties in the dilation theory of CP-semigroups follows simply from the fact that a general CP map on $\sB(H)$ has a form 
\bes
T(a) = \sum_{i=1}^\infty t_i a t_i^* ,
\ees
where $t_1, t_2, \ldots \in \sB(H)$ satisfies $\sum_{i=1}^\infty t_i t_i^* \leq I_H$, rather than the simple form (\ref{eq:analog}). But it is not this multiplicity alone that accounts for the complications. Indeed, let $R,S,T$ be three unital CP maps given by 
\bes
R(a) = \sum_{i=1}^\infty r_i a r_i^* \,\, , \,\, S(a) = \sum_{i=1}^\infty s_i a s_i^* \,\,,\,\, T(a) = \sum_{i=1}^\infty t_i a t_i^* ,
\ees 
for all $a \in \sB(H)$, where $r_1, r_2, \ldots$, $s_1, s_2, \ldots$ and $t_1, t_2, \ldots$ are sequences of operators in $\sB(H)$ satisfying
\be\label{eq:unital}
\sum_{i=1}^\infty r_i  r_i^*  = \sum_{i=1}^\infty s_i  s_i^* =  \sum_{i=1}^\infty t_i  t_i^* = I_H .
\ee
Equation (\ref{eq:unital}) is equivalent to $R$, $S$ and $T$ being unital maps. Assume that for all $i,j$, 
\be\label{eq:commute_op}
r_i s_j = s_j r_i \,\, , \,\, s_i t_j = t_j s_i  \,\, \textrm{ and } \,\, t_i r_j = r_j t_i .
\ee
This implies $R$, $S$ and $T$ commute. One can show, using the methods of \cite{ShalitSolel} (especially \cite[Corollary 5.10]{ShalitSolel}), that $R$, $S$ and $T$ can be dilated to a three-tuple of commuting, unital $*$-endomorphisms on some $\sB(K)$. 
This illustrates that the complications in the dilation theory of CP maps do not follow merely from the multiplicity that CP maps might have. The complications arise from the fact that, in general, the relations between the operators  $r_1, r_2, \ldots$, $s_1, s_2, \ldots$ and $t_1, t_2, \ldots$ giving rise to three commuting CP maps is by far more complicated than (\ref{eq:commute_op}).

\section{Proof of Theorem \ref{thm:main}}

The strategy of the proof of Theorem \ref{thm:main} is roughly as follows. We know from \cite[Section 5.3]{ShalitSolel} that there are examples of nonunital CP-semigroups over $\mb{N}^3$ that have no strong dilation (acting on a type I factor). To prove the existence of a Markov semigroup over $\mb{N}^3$ that has no dilation, we will show that every (nonunital) CP-semigroup $T$ has a unitalization $\widetilde{T}$, such that every dilation of $\widetilde{T}$ gives rise to a strong dilation of $T$. The unitalization of any CP-semigroup that has no strong dilation will therefore be a Markov semigroup that has no dilation.

\begin{proposition}\label{prop:unitalization_B(H)}
Let $T = \{T_s\}_{s\in\mb{S}}$ be a CP-semigroup acting on $\sB(H)$. Then the family $\widetilde{T} = \{\widetilde{T}_s\}_{s \in \mb{S}}$ of linear operators acting on $\sB(H \oplus \mb{C})$ given by $\widetilde{T}_0 = {\bf id}$ and
\bes
\widetilde{T}_s 
\left( 
\begin{array}
[c]{ccc}%
A & h \\
g^* & c
\end{array}
\right) = \left( 
\begin{array}
[c]{ccc}%
T_s(A) + c(I - T_s(I)) & 0 \\
0 & c
\end{array}
\right) \quad, \quad s \neq 0,
\ees
is a semigroup of unital CP-maps such that for all $b \in \sB(H)$ and $s \in \mb{S}$
\be\label{eq:CP0dilCP}
T_s(b) = p_H\widetilde{T}_s(b)p_H.
\ee
\end{proposition}
\begin{proof} 
Straightforward verification.
\end{proof} 
The reader may want to compare this unitalization procedure to the one used in \cite[Section 8]{BS00}.
The advantages of this unitalization are: 1) $\sB(H)$ is a \emph{full} corner of the algebra on which $\widetilde{T}$ acts, and 2) $\widetilde{T}$ acts on a type I factor. The disadvantage is that $\widetilde{T}$ is not continuous at $s = 0$, even when $T$ is, so that this unitalization will not be useful for continuous semigroups.

\begin{theorem}\label{thm:unitalization_dil}
Let $T$ be a CP-semigroup acting on $\sB(H)$, and let $\widetilde{T}$ be the unitalization given by Proposition \ref{prop:unitalization_B(H)}. If $\widetilde{T}$ has a dilation acting on $\widetilde{\cB}$, then $T$ has a strong dilation acting on a corner of $\widetilde{\cB}$.
\end{theorem}
\begin{proof}
Denote by $(\tilde{p},\widetilde{K},\widetilde{\cB},\tilde{\theta})$ the strong dilation of $\widetilde{T}$. Since $\sB(H \oplus \mb{C})$ is embedded as $\widetilde{p}\widetilde{\cB}\widetilde{p}$ in $\widetilde{\cB}$, we have the identification $\widetilde{p} = 1_{\sB(H \oplus \mb{C})}$. Write $1$ for $1_{\widetilde{\cB}}$. We define a strong dilation $(p,K,\cB,\theta)$ for $T$ as follows. First, in $\sB(H \oplus \mb{C}) \subseteq \widetilde{\cB}$, define $p = I_H \oplus 0$ and $q = 0 \oplus 1_\mb{C}$. The Hilbert space $K$ is then defined as $K = (1-q) \widetilde{K}$. Next, define 
\bes
\cB = (1-q)\widetilde{\cB} (1-q).
\ees
Writing $1_\cB$ for the identity of $\cB$, we have that $1_\cB = 1-q$ and $\cB = 1_\cB \widetilde{\cB} 1_\cB$. We claim that $\cB$ is invariant for $\tilde{\theta}$. To see this, note first that $\tilde{\theta}_s(1_\cB) \leq 1_\cB$. Indeed, $q \tilde{\theta}_s(q) q = q \tilde{p} \tilde{\theta}_s(q) \tilde{p} q = q \widetilde{T}_s(q) q = q$, thus $\tilde{\theta}_s(q) \geq q$, whence
\bes
\tilde{\theta}_s(1_{\cB}) = \tilde{\theta}_s(1-q) = \tilde{\theta}(1) - \tilde{\theta}_s(q) \leq 1 - q = 1_{\cB}.
\ees
It follows that for every $b = 1_{\cB} b 1_{\cB} \in \cB$, 
\bes
\tilde{\theta}_s(b) = \tilde{\theta}_s(1_{\cB} b 1_{\cB}) = \tilde{\theta}_s(1_{\cB}) \tilde{\theta}_s(b) \tilde{\theta}_s(1_{\cB}) = 1_\cB \tilde{\theta}_s(1_{\cB}) \tilde{\theta}_s(b) \tilde{\theta}_s(1_{\cB})1_\cB \in \cB.
\ees
That proves the claim that $\cB$ is invariant for $\tilde{\theta}$. 

Now we may define an E-semigroup $\theta$ acting on $\cB$ by $\theta_s = \tilde{\theta}_s\big|_{\cB}$. It remains to show that $(p,K,\cB,\theta)$ is a strong dilation of $T$. To that end, we first show that $p \in \cB$ is coinvariant for $\theta$, that is, that for all $s \in \mb{S}$, 
\be\label{eq:coinvariant}
\theta_s(1_\cB - p) \leq 1_\cB - p. 
\ee 
But
\begin{align*}
p \theta_s(1_\cB - p)p 
&= p \tilde{p} \tilde{\theta}_s(1 - q - p)\tilde{p} p \\
&= p \tilde{p} \tilde{\theta}_s(1 - \tilde{p})\tilde{p} p \\
&= 0 , 
\end{align*}
because $\tilde{p}$ is coinvariant for $\tilde{\theta}$ (that is, $\tilde{\theta}_s(1-\tilde{p}) \leq 1 - \tilde{p}$ for all $s \in \mb{S}$. This follows directly from the dilation condition: $\tilde{p} \tilde{\theta}_s (1-\tilde{p}) \tilde{p} = \widetilde{T}_s(\tilde{p}(1-\tilde{p})\tilde{p}) = 0$). This establishes (\ref{eq:coinvariant}). Now we can show that $\theta$ is a strong dilation for $T$. First note that (\ref{eq:coinvariant}) implies that for every $b \in \cB$,
\bes
p \theta_s(b) p - p \theta_s(pb)p = p\theta_s(1_\cB - p) \theta_s(b) p = 0,
\ees
thus $p \theta_s(b) p = p \theta_s(pb)p$. Taking adjoints and using $pb$ instead of $b$ we get
\bes
p \theta_s(b) p = p \theta_s(pbp)p.
\ees
Therefore for all $b \in \cB$
\begin{align*}
p \theta_s(b) p 
&= p \theta_s(pbp)p \\
&= p\tilde{p} \tilde{\theta}_s(pbp)\tilde{p}p \\
&= p \widetilde{T}_s(pbp) p \\
&= T_s(pbp).
\end{align*}
\end{proof}

Before completing the proof of Theorem \ref{thm:main} we need to make the following preparations. The following lemma is stated in greater generality than we need, because it is interesting.
\begin{lemma}\label{lem:min_equiv}
Let $T$ be a CP-semigroup on $\cA$, and let $(p,K,\cB,\theta)$ be a dilation. Then $(p,K,\cB,\theta)$ is minimal if and only if 
\begin{enumerate}
\item $\cB$ is the W$^*$-algebra generated by $\cup_{s\in\mb{S}}\theta_s(\cA)$; and
\item $K = \overline{\spn}\{\theta_{s_1}(a_1) \cdots \theta_{s_k}(a_k) h : h \in H, a_i \in \cA_i, s_i \in \mb{S}, 1\leq i\leq k\}.$
\end{enumerate}
\end{lemma}
\begin{proof}
This follows immediately from the fact that the central carrier of $p$ has range $[\cB H]$ (see \cite[Proposition 5.5.2]{KR}).
\end{proof}

\begin{lemma}\label{lem:minB(K)}
Let $T = \{T_s\}_{s\in\mb{S}}$ be a CP-semigroup on $\sB(H)$, and let $(p,K,\cB,\theta)$ be a minimal dilation for $T$. Then $\cB = \sB(K)$.
\end{lemma}
\begin{proof}
The proof is standard (see \cite[Proposition 6.8]{Shalit08} or \cite[pp. 292--293]{Arv03}) and is included for completeness.

By Lemma \ref{lem:min_equiv} we have
\bes
K  = \overline{\spn}\{\theta_{s_1}(a_1) \cdots \theta_{s_k}(a_k) h : h \in H, a_i \in \sB(H), s_i \in \mb{S}, 1\leq i\leq k\}.
\ees
Now let $q \in \cB'$. Then $qp = pq = pqp$ is a projection in $\sB(H)$ that commutes with $\sB(H)$, thus either $qp = 0$ or $qp = p = I_H$. If $qp = 0$, then 
\bes
q \theta_{s_1}(a_1) \cdots \theta_{s_k}(a_k) h = \theta_{s_1}(a_1) \cdots \theta_{s_k}(a_k) qp h =0,
\ees
so $q = 0$. If $qp = p$ then 
\bes
q \theta_{s_1}(a_1) \cdots \theta_{s_k}(a_k) h = \theta_{s_1}(a_1) \cdots \theta_{s_k}(a_k) qp h = \theta_{s_1}(a_1) \cdots \theta_{s_k}(a_k) h ,
\ees
so $q = I_K$. We conclude that $\cB' = \mb{C}\cdot I_K$, whence $\cB = \cB'' = \sB(K)$.
\end{proof}

\noindent{\em Completion of the proof of Theorem \ref{thm:main}}.
Let $T$ be a CP-semigroup over $\mb{N}^3$ acting on $\sB(H)$ that has no strong dilation acting on a type I factor. Two such examples were constructed in \cite[Section 5.3]{ShalitSolel}. The example in \cite[Theorem 5.14]{ShalitSolel} acts on a $\sB(H)$ with $\dim H = 5$. Let $\widetilde{T}$ be the unitalization of $T$ given by Proposition \ref{prop:unitalization_B(H)} (so $\widetilde{T}$ acts on $\sB(\widetilde{H})$ with $\dim \widetilde{H} = 6$). We claim that $\widetilde{T}$ has no minimal dilation, and in fact no dilation on a type I factor whatsoever. 

We first show that $\widetilde{T}$ has no dilation acting on a type I factor. Assume to the contrary that $\widetilde{T}$ has a dilation $(\tilde{p},\widetilde{K},\sB(\widetilde{K}),\tilde{\theta})$. By Theorem \ref{thm:unitalization_dil}, this implies that $T$ has a strong dilation acting on some corner of $\sB(\widetilde{K})$ - therefore a type I factor. This contradicts the choice of $T$.

Finally, to see that $\widetilde{T}$ has no {\em minimal} dilation, we use Lemma \ref{lem:minB(K)}, which tells us that if $(\tilde{p},\widetilde{K},\widetilde{\cB},\tilde{\theta})$ is a \emph{minimal} dilation of $\widetilde{T}$, then $\widetilde{\cB} = \sB(\widetilde{K})$. By the previous paragraph, this is impossible.

\section{Open problems}

There are several interesting questions that are left open:
\begin{enumerate}
\item The counter example we gave acts on $\sB(H)$ with $\dim H =6$. {\em What is the smallest dimension of $H$ such for which such a counter example can be found?}

\item We did not exclude the possibility that every three-tuple of commuting, unital CP maps has a strong dilation, if one also allows for non-minimal dilations and arbitrary dilating von Neumann algebras. {\em Is there a three tuple of CP maps that has no dilation whatsoever?}

\item {\em What conditions guarantee the existence of a (strong) dilation for a Markov or a CP-semigroup?}
\end{enumerate}
A better understanding of these problems will involve {\em subproduct systems}. We plan to treat some of these problems in the future.

\subsection*{Acknowledgments.} The authors thank the anonymous referee for his remarks. This research was done partially while the second author was visiting the first author at the Department of Pure Mathematics in the University of Waterloo in November 2009, and while the first author was visiting the second one at the Department of Mathematics and Statistics in Queen's University (Kingston) in January 2010. The generous and warm hospitality provided by the departments and by the hosts Ken Davidson and Roland Speicher is greatly appreciated. The second author is supported by research funds of the Italian MIUR (under PRIN 2007) and of the University of Molise. This includes also partial support of the visit in January.

\bibliographystyle{amsplain}

\end{document}